\documentclass{article}
\usepackage{amssymb}
\title{{\bf The Complexity of Power Graphs Associated With Finite Groups}}
\author{{\bf S. Kirkland}$^1$, {\bf A. R.  Moghaddamfar}$^{2, 3}$, {\bf S. Navid Salehy}$^{4}$,\\[0.2cm]  {\bf S. Nima Salehy}$^{4}$ and {\bf M. Zohourattar}$^2$\\[0.3cm]
{\em $^1$Department of Mathematics, University of Manitoba,}\\
{\em Winnipeg, MB, Canada.}\\[0.2cm]
{\tt Stephen.Kirkland@umanitoba.ca}\\[0.3cm] 
{\em $^2$Faculty of Mathematics, K. N. Toosi
University of Technology,}\\
{\em P. O. Box $16315$--$1618$, Tehran, Iran,}\\[0.2cm]
{\em $^3$Department of Mathematical Sciences, Kent State
University,}\\ {\em  Kent, Ohio $44242$,  USA}\\[0.2cm]
 {\tt
moghadam@kntu.ac.ir}, and {\tt amoghadd@kent.edu}\\[0.3cm] 
{\em $^4$Department of Mathematics, Florida State University,}\\[0.2cm]
{\em Tallahassee, FL $32306$, USA.}\\[0.1cm]
{\tt navidsalehy@math.fsu.edu}, and 
{\tt
nimasalehy@math.fsu.edu}}

\newenvironment{proof}{\noindent {\em {Proof}}.}{$\square$
\medskip}
\newtheorem{theorem}{Theorem}[section]
\newtheorem{definition}[theorem]{Definition}
\newtheorem{corollary}[theorem]{Corollary}

\newtheorem{remark}[theorem]{Remark}

\newtheorem{lm}[theorem]{Lemma}

\begin{document}
\newcommand{\f}{\frac}
\newcommand{\sta}{\stackrel}
\maketitle
\begin{abstract}
\noindent   The power graph $\mathcal{P}(G)$ of a finite group $G$ is the graph whose vertex set is $G$, and two elements in $G$ are adjacent if one of them is a power of the other.  
The purpose of this paper is twofold. First, we find the complexity of a clique--replaced graph and study  
 some applications. Second, we derive some explicit formulas concerning the complexity $\kappa(\mathcal{P}(G))$ 
for various groups $G$ such as the cyclic group of order $n$, the simple groups $L_2(q)$, the extra--special $p$--groups of order $p^3$, 
 the Frobenius groups, etc. 
\end{abstract}


\renewcommand{\baselinestretch}{1}
\def\thefootnote{ \ }
\footnotetext{{\em $2010$ Mathematics Subject Classification}: 20D06, 05C05, 05C50.}
\section{Introduction}
All graphs considered here are simple connected graphs.
A {\em spanning tree} of a  connected graph
 is  a subgraph that contains all the vertices and is a tree.
 Counting the number of spanning trees in a connected graph
 is a problem of long--standing interest
 in various fields of science.
 For a  graph $\Gamma$, the 
 number of spanning trees of $\Gamma,$  denoted by $\kappa(\Gamma),$ is known as the {\em complexity} of $\Gamma$.

In this paper, we consider some graphs arising from finite groups. One well--known graph is the power graph, as defined more precisely below.
\begin{definition} {\rm  Let $G$ be a finite group and $X$ a nonempty subset of $G$.
The {\em power graph} $\mathcal{P}(G, X)$,  has
$X$ as its vertex set and two vertices $x$ and $y$ in $X$ are joined by an edge if $\langle
x\rangle\subseteq \langle y\rangle$ or $\langle y\rangle\subseteq
\langle x\rangle$.  }
\end{definition}

The term {\em power graph} was introduced in \cite{Kelarev2}, and after that power graphs have been investigated by many authors, see for instance \cite{Kelarev, Cameron, Moghaddamfar}. 
The investigation of power graphs associated with algebraic structures is important, because these graphs have valuable applications (see the survey article \cite{Kelarev3}) and are related to automata theory (see the book \cite{Kelarev1}).

In the case $X=G$, we will simply write  $\mathcal{P}(G)$ instead of $\mathcal{P}(G, G)$.
Clearly, when $1\in X$,  the power graph is connected, and we can talk about the complexity 
of this graph. For convenience, we put  $\kappa_G(X)=\kappa(\mathcal{P}(G, X))$ and $\kappa(G)=\kappa(\mathcal{P}(G))$. A well known result due to Cayley \cite{Cayley} says that the complexity of the complete graph on $n$ vertices is $n^{n-2}$. In \cite{Chakrabarty}  it was shown that 
a finite group has a complete power graph if and only if it is a cyclic $p$--group, where $p$ is a prime
number.
Thus, as an immediate consequence of Cayley's result, we derive $\kappa({\Bbb Z}_{p^m})=p^{m(p^m-2)}$.  Recently, the authors of \cite{MRSS}   obtained  a formula to compute the complexity $\kappa({\Bbb Z}_{n})$ for any $n$ (see Corollary \ref{result2} below). To obtain Corollary \ref{result2}, we will define a class of graphs more general than the power graphs of cyclic groups. Specifically, we start with a graph $\Gamma$ on  vertices  $v_1, v_2, \ldots, v_n$.  To construct a new graph, we replace each $v_i$  by a complete graph $K_{x_i}$ on $x_i$ vertices,  and if there is an edge between $v_i$ and $v_j$ in $\Gamma$, then  we connect each vertex of $K_{x_i}$  with each vertex of $K_{x_j}$.  The new graph will be denoted by $\Gamma_{[x_1, \ldots, x_n]}$.  We will derive explicit formulas for the complexity $\kappa(\Gamma_{[x_1, \ldots, x_n]})$ (see Theorem \ref{result1} and Remark \ref{alt_appch}). Then we will obtain a formula for the complexity  $\kappa({\Bbb Z}_{n})$  by choosing  a certain graph $\Gamma$  on $k$ vertices and positive integers $x_1, x_2, \ldots, x_k$  (Corollary \ref{result2}).  Finally, the complexities $\kappa(G)$ for certain groups $G$  are presented.

The outline of the paper is as follows. In the next section, we recall 
 some basic definitions and notation and give several auxiliary
results to be used later. The main result
of Section 4 is Theorem \ref{result1} and we include some of its applications. 
In Section 5, we compute  $\kappa(G)$ for certain groups $G$.


\section{Terminology and Previous Results}
We first establish some notation which will be used repeatedly in the sequel.
Given a graph $\Gamma$, we denote by $\mathbf{A}_\Gamma$ and $\mathbf{D}_\Gamma$ the adjacency matrix and the diagonal matrix of vertex degrees of $\Gamma$, respectively. The Laplacian matrix of $G$ is defined
as $\mathbf{L}_\Gamma=\mathbf{D}_\Gamma-\mathbf{A}_\Gamma$. Clearly, $\mathbf{L}_\Gamma$ is a real symmetric matrix and its eigenvalues are nonnegative real numbers.
The Laplacian spectrum of $\Gamma$ is
$${\rm Spec}(\mathbf{L}_\Gamma)=\left(\mu_1(\Gamma), \mu_2(\Gamma), \ldots, \mu_n(\Gamma)\right),$$
where $\mu_1(\Gamma)\geqslant \mu_2(\Gamma)\geqslant \cdots\geqslant \mu_n(\Gamma)$,
are the eigenvalues of $L_\Gamma$ arranged in weakly decreasing order, and $n=|V(\Gamma)|$.
Note that  $\mu_n(\Gamma) = 0$, because each row sum of $\mathbf{L}_\Gamma$ is 0.
Instead of $\mathbf{A}_\Gamma$, $\mathbf{L}_\Gamma$, and $\mu_i(\Gamma)$  we simply write $\mathbf{A}$, $\mathbf{L}$,  and $\mu_i$  if it does not lead to confusion.
Given a subset $\Lambda$ of the vertex set of a graph, we let $\mathbf{A}(\Lambda)$ denote the principal submatrix of $\mathbf{A}$ corresponding to the vertices in $\Lambda$. 

For a graph with $n$ vertices and Laplacian spectrum $\mu_1\geqslant \mu_2\geqslant \cdots\geqslant \mu_n=0$
it has been proved \cite[Corollary 6.5]{Biggs} that:
\begin{equation}\label{eq1}
\kappa(\Gamma)=\frac{\mu_1 \mu_2 \cdots \mu_{n-1}}{n}.
\end{equation}

The vertex--disjoint union of the graphs $\Gamma_1$ and $\Gamma_2$ is denoted by $\Gamma_1\oplus\Gamma_2$.
Define the {\em join} of $\Gamma_1$ and $\Gamma_2$ to be $\Gamma_1 \vee \Gamma_2=(\Gamma_1^c\oplus\Gamma_2^c)^c$. Evidently this is the graph formed from the vertex--disjoint union of the two graphs $\Gamma_1, \Gamma_2$,
by adding edges joining every vertex of  $\Gamma_1$  to every
vertex of $\Gamma_2$.
Now, one may easily prove the following (see also \cite{Merris}). 
\begin{lm}\label{elementary0}
Let $\Gamma_1$ and $\Gamma_2$ be two graphs on disjoint sets  with $m$ and $n$ vertices, respectively. If
$${\rm Spec}(\mathbf{L}_{\Gamma_1})=\left(\mu_1(\Gamma_1), \mu_2(\Gamma_1), \ldots, \mu_m(\Gamma_1)\right),$$
and
$${\rm Spec}(\mathbf{L}_{\Gamma_2})=\left(\mu_1(\Gamma_2), \mu_2(\Gamma_2), \ldots, \mu_n(\Gamma_2)\right),$$ then, the following hold:
\begin{itemize}
\item[{\rm (1)}] the eigenvalues of Laplacian matrix  $\mathbf{L}_{\Gamma_1\oplus \Gamma_2}$ are:
$$\mu_1(\Gamma_1), \  \ldots, \ \mu_{m}(\Gamma_1),  \ \mu_1(\Gamma_2), \  \ldots, \ \mu_{n}(\Gamma_2).$$

\item[{\rm (2)}] the eigenvalues of Laplacian matrix  $\mathbf{L}_{\Gamma_1\vee \Gamma_2}$ are:
$$m+n, \ \mu_1(\Gamma_1)+n, \  \ldots, \  \mu_{m-1}(\Gamma_1)+n, \ \mu_1(\Gamma_2)+m, \ \ldots, \ \mu_{n-1}(\Gamma_2)+m, \ 0.$$
\end{itemize}
\end{lm}

A {\em universal} vertex is a vertex of a graph that is adjacent to all other vertices of the graph. Now, we restrict our attention to information about the set of universal 
vertices of the power graph of a group $G$. As already mentioned, the identity element of $G$ is  a universal vertex in $\mathcal{P}(G)$, and also $\mathcal{P}(G)$ is complete if and only if $G$ is cyclic of prime power order, and in this case $G$ is the set of all universal vertices. However, the following lemma  \cite[Proposition $4$]{Cameron} determines the set of universal vertices of the power graph
of $G$, in the general case.

\begin{lm}\label{universal} 
 Let $G$ be a finite group and $S$  the set of universal vertices of the power graph $\mathcal{P}(G)$. Suppose that $|S|>1$. Then one of the following occurs:
\begin{itemize}
\item[{\rm (a)}]  $G$ is cyclic of prime power order, and $S=G$;
\item[{\rm (b)}]  $G$ is cyclic of non--prime--power order $n$, and $S$ consists of the identity and the 
generators of $G$, so that $|S|=1+\phi(n)$, where $\phi$ is Euler's $\phi$--function;
\item[{\rm (c)}]  $G$ is generalized quaternion, and $S$ contains the identity and the unique involution in $G$, so that $|S|=2$.
\end{itemize}
\end{lm}

We conclude this section with notation and definitions to be used in 
the paper. All the groups considered here are finite.  We denote by $[G,G]$ the commutator subgroup, for any group $G$. If $g \in G$, then $o(g)$ denotes the order of the element $g$. We refer to any  element in $G$  of order $2$, as an {\em involution}. 
An {\em elementary abelian $p$--group} of order $p^n$, denoted by  ${\Bbb E}_{p^n}$, is isomorphic to a direct 
product of $n$ copies of the cyclic group ${\Bbb Z}_p$. The complement of a graph $\Gamma$ is denoted by $\Gamma^c$. The neighborhood of a vertex $v$ in the graph $\Gamma$ is denoted by $N_\Gamma(v)$. Let $K_n$ denote the {\em complete graph} ({\em clique}) with $n$ vertices.
 Throughout we use the standard notation and terminology introduced in \cite{Biggs, Isaacs} for graph theory and group theory.


\section{Auxiliary Results}
\begin{lm}\label{join-integer}  Let $\Gamma$ be any graph on $n$ vertices with Laplacian spectrum
$$\mu_1 \geqslant \mu_2 \geqslant  \cdots \geqslant  \mu_n.$$
If $m$ is an integer, then the following product
$$(\mu_1+m)(\mu_2+m)\cdots (\mu_{n-1}+m),$$
is also an integer.
\end{lm}
\begin{proof}  Consider the characteristic polynomial of the Laplacian matrix $\mathbf{L}$:
$$\sigma(\Gamma; \mu)=\det (\mu \mathbf{I}-\mathbf{L})=\mu^n+c_1\mu^{n-1}+\cdots+c_{n-1}\mu+c_n.$$
First, we observe that the coefficients $c_i$ are integers \cite[Theorem 7.5]{Biggs}, and in particular, $c_n=0$.
This forces  $\sigma(\Gamma; -m)$ to be an integer, which is  divisible by $m$.
Moreover, we have
$$\sigma(\Gamma; \mu)=(\mu-\mu_1)(\mu-\mu_2)\cdots (\mu-\mu_n),$$
and since $\mu_n=0$, we obtain
$$\sigma(\Gamma; -m)=(-1)^n m(\mu_1+m)(\mu_2+m)\cdots (\mu_{n-1}+m).$$
 The result now follows.
\end{proof}

\begin{lm}\label{full}
Let  a graph $\Gamma$ with $n$ vertices contain $m<n$ universal vertices. Then $k(\Gamma)$ is divisible
by $n^{m-1}$.
\end{lm}
\begin{proof}  Let $W$ be the set of universal vertices, $\Gamma_0=\Gamma-W$ and $t=n-m$. 
Clearly, we have $\Gamma=K_m\vee \Gamma_0$. Let 
$\mu_1 \geqslant \mu_2 \geqslant  \cdots \geqslant  \mu_t=0$, be the eigenvalues of $\mathbf{L}_{\Gamma_0}$.  
Since the Laplacian matrix for the complete graph $K_m$ has eigenvalue $0$ with multiplicity $1$ and eigenvalue $m$ with multiplicity $m-1$,  it follows by Lemma \ref{elementary0}  that  the eigenvalues of the Laplacian matrix  $\mathbf{L}_{\Gamma}$ are:
$$n, \underbrace{n, \ n, \ \ldots, \ n}_{m-1}, \  \underbrace{\mu_1+m, \ \mu_2+m, \ \ldots, \ \mu_{t-1}+m}_{t-1}, \ 0. $$
We find immediately using Eq. (\ref{eq1}) that
$$\kappa(\Gamma)=n^{m-1}(\mu_1+m)(\mu_2+m)\cdots (\mu_{t-1}+m).$$
Finally, since $(\mu_1+m)(\mu_2+m)\cdots (\mu_{t-1}+m)$ is an integer by Lemma \ref{join-integer},  we obtain the result. \end{proof}

Let $Q_{2^n}$ $(n\geqslant 3)$ denote the generalized quaternion
group of order $2^n$, which can be presented by
$$Q_{2^n}=\langle x, y \ | \ x^{2^{n-1}}=1, y^2=x^{2^{n-2}}, x^y=x^{-1}\rangle.$$
Moreover, the power graph $\mathcal{P}(Q_{2^n})$ has the following form:
$$\mathcal{P}(Q_{2^n})=K_2\vee \Big(K_{2^{n-1}-2}\oplus
\underbrace{K_2\oplus K_2\oplus \cdots \oplus K_2}_{2^{n-2} {\rm
-times}}\Big).$$ 
Using Lemma \ref{elementary0} and Eq. (\ref{eq1}), we have the following corollary \cite[Theorem 5.2]{MRSS}:
\begin{corollary}\label{cor-quater} Let $n\geqslant 3$ be an integer. Then,  
$\kappa(Q_{2^n})=2^{(2^{n-2}-1)(2n+1)+4}$.
\end{corollary}

A  finite group $G$ is called an {\em element prime order} group (EPO--group) if every nonidentity element of $G$ has prime order. We can consider the power graph of an EPO--group $G$ as follows: 
$$ \mathcal{P}(G) =K_1 \vee \left(\bigoplus_{p\in \pi(G)} c_pK_{p-1}\right),$$
where $c_p$ signifies the number of cyclic subgroups of order $p$
in $G$. Again, using Lemma \ref{elementary0} and Eq. (\ref{eq1}), we have the following corollary \cite[Corollary 3.4]{MRSS}:
\begin{corollary}\label{cor-epo} Let $G$ be an EPO--group. Then we have:
$$\kappa(G)=\prod_{p\in \pi(G)}p^{(p-2)c_p}.$$
In particular,  we have 
$$\kappa\left({\Bbb E}_{p^n}\right)=p^{(p-2)(p^n-1)/(p-1)}.$$
\end{corollary}


\section{Clique--Replaced Graphs}
Let $\Gamma$ be a connected graph with vertices $v_1, \ldots, v_k$. Given positive integers $x_1, \ldots, x_k$,
we construct the new graph $\Gamma_{[x_1, \ldots, x_k]}$ as follows: Replace vertex $v_i$ 
in $\Gamma$ by the complete graph (clique) $K_{x_i}$, $i=1, \ldots, k$, and label the vertex set of $K_{x_i}$ for each $i$ as:
$u_{i_1}, u_{i_2}, \ldots, u_{i_{x_i}}$. Now, if $v_i$ is adjacent to $v_j$ in $\Gamma$, then connect all vertices 
$u_{i_1}, u_{i_2}, \ldots, u_{i_{x_i}}$ with all vertices $u_{j_1}, u_{j_2}, \ldots, u_{j_{x_j}}$. 
We call the resulting graph $\Gamma_{[x_1, \ldots, x_k]}$ the {\em clique--replaced graph}.
It is clear that for a fixed $i$, all vertices $u_{i_1}, u_{i_2}, \ldots, u_{i_{x_i}}$ have the same degree which is equal to 
$$n_i=x_i-1+\sum_{v_j\in N_\Gamma(v_i)} x_j.$$
Put 
$m_i=n_i+1=x_i+\sum_{v_j\in N_\Gamma(v_i)} x_j, \ \ \ \lambda_i=\frac{m_i}{x_i}, \ \ i=1, \ldots, k,$
and $\Psi=\prod_{i=1}^{k} \lambda_i.$ 
Suppose that $n=x_1+\cdots+x_k$. 

\begin{theorem}\label{result1}  With the notation as explained above, we have 
\begin{equation}\label{formula1} \kappa\left(\Gamma_{[x_1, \ldots, x_k]}\right)=\prod_{i=1}^{k}m_i^{x_i}\Big(\Psi+
\sum_{\Lambda} \det \mathbf{A}_{\Gamma^c}(\Lambda)\lambda_1^{t_1}
\lambda_2^{t_2}\cdots \lambda_{k}^{t_{k}}\Big)/(\Psi
n^2),\end{equation} where $t_i\in \{0, 1\}$, $i=1,  \ldots,
k$, and the summation is over all induced subgraphs $\Lambda$ of
$\Gamma^c$ whose vertex set $\{v_{i_1}, \ldots , v_{i_s} \}$ corresponds to 
  $\{i_j| t_{i_j}=0 \}.$  
\end{theorem}
\begin{proof}$^{1}$\footnote{$^{1}$The idea of this proof is borrowed from \cite[Theorem 4.1]{MRSS}.}  Let $\Gamma^\ast=\Gamma_{[x_1, \ldots, x_k]}$.
 It is easy to check that, the matrix
$\mathbf{J}+\mathbf{L}_{\Gamma^\ast}$ associated with $\Gamma^\ast$ has the following block--matrix structure:
\begin{equation}\label{e1} \mathbf{J}+\mathbf{L}_{\Gamma^\ast}= \left(\mathbf{D}_{ij}\right)_{1\leqslant i,j\leqslant k},
\end{equation} where $\mathbf{D}_{ij}$ is a matrix of size $x_i\times x_j$
with
$$D_{ij}=\left\{\begin{array}{lll} m_i\mathbf{I} & if & i=j,\\[0.2cm]
0 & if & i\neq j, \ \  v_i\sim v_j \ \mbox{in} \ \Gamma,\\[0.2cm]
\mathbf{J} &  & \mbox{otherwise}.
\end{array} \right.$$
We need only to evaluate $\det (\mathbf{J}+\mathbf{L}_{\Gamma^\ast})$, because $\kappa(\Gamma^\ast)=\det(\mathbf{J}+\mathbf{L}_{\Gamma^\ast})/n^2$.
In what follows, $D$ denotes the determinant of the matrix on the right--hand
side of Eq. (\ref{e1}). In order to compute this determinant, we
apply the following row and column operations:
We subtract column $j$ from column $j+r$:
$$\left\{\begin{array}{ll} j=1+\sum\limits_{l=1}^{h}x_l, \ h=0, 1,2, \ldots, k-1,\\[0.3cm]
r=1, 2, \ldots, x_{h+1}-1.\end{array} \right.$$
Then, we add row $i+s$ to row $i$:
$$\left\{\begin{array}{ll} i=1+\sum\limits_{l=1}^{h}x_l, \ h=0, 1,2, \ldots, k-1,\\[0.3cm]
s=1, 2, \ldots, x_{h+1}-1.\end{array} \right.$$ (Note that,
when $m>n$, we adopt the convention that $\sum_{i=m}^n x_i=0$.) Using the above
operations, it is easy to see that
$$D=\det \left(\mathbf{M}_{ij}\right)_{1\leqslant i,j\leqslant k},$$
where $\mathbf{M}_{ij}$ is a matrix of size $x_i\times x_j$
with
$$ \mathbf{M}_{ij}=\left\{\begin{array}{ll} m_i\mathbf{I} &  \mbox{if} \  \  i=j,\\[0.2cm]
0 &  \mbox{if} \ \  i\neq j, \ \  v_i\sim v_i \ \ \mbox{in} \ \Gamma,\\[0.2cm]
x_i\mathbf{E}_{1, 1}+\mathbf{E}_{2,1}+\cdots+\mathbf{E}_{x_i, 1} & \mbox{otherwise},
\end{array} \right.$$
where $\mathbf{I}$ is the identity matrix and $\mathbf{E}_{i,j}$
denotes the square matrix having $1$ in the $(i,j)$ position and
$0$ elsewhere.

Therefore, taking out the common factors and developing the
determinant along the columns $j$,
$j\neq1+\sum\limits_{l=1}^{h}x_l$, $h=0, 1, 2, \ldots, k-1$, one
gets
\begin{equation}\label{e2} D=\Phi^{-1}\prod_{i=1}^{k}{m_i}^{x_i}\cdot\det
\left(c_{ij}\right)_{1\leqslant i,j\leqslant k},\end{equation}
where $$ c_{ij}=\left\{\begin{array}{cl}\lambda_i &  \mbox{if} \  \  i=j,\\[0.2cm]
0 &  \mbox{if} \ \  i\neq j, \ \ v_i\sim v_j \  \mbox{in} \ \Gamma,\\[0.2cm]
1 & \mbox{otherwise}.
\end{array} \right.$$
As the reader might have noticed, the matrix
$\left(c_{ij}\right)_{1\leqslant i,j\leqslant k}-{\rm diag}(\lambda_1, \lambda_2, \ldots, \lambda_k)$
is exactly the adjacency matrix of the graph
$\Gamma^c$. Consequently, we
get
$$\det \left( \left[
\begin{array}{cccc}
\lambda_1& c_{12} & \ldots & c_{1k} \\[0.1cm]
c_{21} & \lambda_2 & \ldots & c_{2k} \\[0.1cm]
\vdots & \vdots & \ddots & \vdots \\[0.1cm]
c_{k1} & c_{k2} & \ldots & \lambda_{k} \\[0.1cm]
\end{array}\right ] \right) =\Psi+\sum_{\Lambda}
\det \mathbf{A}_{\Gamma^c}(\Lambda)\lambda_1^{t_1}\lambda_2^{t_2}\cdots
\lambda_{k}^{t_{k}},$$ where $t_i\in \{0, 1\}$, $i=1, 2,
\ldots, k$,  and the summation is over all induced subgraphs
$\Lambda$ of $\Gamma^c$ whose  vertex set $\{v_{i_1}, \ldots,
v_{i_s}\}$ corresponds to 
  $\{i_j| t_{i_j}=0 \}.$  
 This is
substituted in Eq. (\ref{e2}):
$$D=\Psi^{-1}\prod_{i=1}^{k}{m_i}^{x_i}\cdot \Big(\Psi+\sum_{\Lambda}
\det \mathbf{A}_{\Gamma^c}(\Lambda)\lambda_1^{t_1}\lambda_2^{t_2}\cdots
\lambda_{k}^{t_{k}}\Big).$$  \end{proof}

\begin{remark} \label{alt_appch} 
{\rm{In this remark, we describe an alternate approach to the computation of $\kappa\left(\Gamma^\ast \right).$ 
Note that $\mathbf{L}_{\Gamma^\ast}$ can be written as a $k \times k$ block matrix, where, for distinct $i,j=1,\ldots, k,$  the $(i,j)$ off--diagonal block is either $-J$ or $0$ according as $v_i$ is adjacent to $v_j,$ or not, and where the $j$--th diagonal block is $m_jI-J, j=1, \ldots, k.$  From this block structure, 
and applying the technique of equitable partitions (see \cite{BH}), 
it follows readily that the product of the nonzero eigenvalues of $\mathbf{L}_{\Gamma^\ast}$ is equal to $ \alpha_2\cdots \alpha_k \left(\prod_{j=1}^{k}\left( m_j \right)^{x_j-1}\right),$ where $0, \alpha_2, \alpha_3, \ldots, \alpha_k$ are the eigenvalues of the 
the $k \times k$ matrix $\mathbf{S}$ whose entries are given by 
$$s_{pq}=\left\{\begin{array}{lll} 0  &  &  \mbox{if $v_p\neq v_q$ and $v_p$ is not adjacent to $v_q$ in $\Gamma$,} \\[0.2cm]
 -x_q & &\mbox{if $p < q$ and $v_p$ is adjacent to $v_q$ in $\Gamma$},\\[0.2cm]
 -x_p & &\mbox{if $p > q$ and $v_p$ is adjacent to $v_q$ in $\Gamma$},\\[0.2cm]
-\sum_{j\neq p} s_{pj} &  & \mbox{if $p=q$}.\\
\end{array} \right.$$
(Observe that $\mathbf{S}$ is singular since every row sums to $0.$) In order to complete the computation of the complexity of 
$L_{\Gamma^\ast}$, we need to find the product $\alpha_2\cdots \alpha_k$. 
For each $j=1, \ldots, k$, let $\mathbf{S}_{(j)}$ denote the principal sub--matrix of $\mathbf{S}$ formed by deleting 
row $j$ and column $j$. Then $\alpha_2\cdots \alpha_k=\sum_{j=1}^{k} \det(\mathbf{S}_{(j)}).$  

To compute the quantities $\det(\mathbf{S}_{(j)}), j=1, \ldots, k,$ we consider a weighted directed graph $W$ on  vertices $1, \ldots, k,$ whose construction we now describe. 
Begin with the graph $\Gamma$ on  vertices  $v_1, v_2, \ldots, v_k$. 
For each edge $\{v_i, v_j\}$ of $\Gamma$, $W$ contains the arcs $i\rightarrow j$ and $j\rightarrow i$;
if $i<j$, the weight of the arc $i\rightarrow j$ in $W$ is $w(i,j)=x_j$, and the  
weight of the arc $j\rightarrow i$  is $w(j,i)=x_i$; if there is no edge between $v_i$ and $v_j$ in $\Gamma$, 
then $W$ contains neither an arc from $i$ to $j$ nor an arc from $j$ to $i$. 
Fix an index $j$ with $1\leqslant j\leqslant k$. We can find $\det(\mathbf{S}_{(j)})$ from  
 a generalization of the matrix tree theorem as follows (see \cite{Chaiken}). Let $\tau_j$ be the set of all spanning directed subgraphs of $W$ such
that: 

(a) the underlying spanning subgraph is a tree; and 

(b) in the spanning directed subgraph of $W$,  
for each vertex $i\neq j$, there is a directed path from $i$ to $j$. 

For each directed graph $\tau\in \tau_j$, let the weight of $\tau$, $\sigma(\tau)$, be the 
product of the weights of the arcs in $\tau$ (these arc weights are inherited from $W$). Then 
$ \det(\mathbf{S}_{(j)})=\sum_{\tau\in \tau_j} \sigma(\tau).$
So, we have 
$\alpha_2\cdots \alpha_k=\sum_{j=1}^{k}  \sum_{\tau\in \tau_j} \sigma(\tau).$
Consequently, we obtain the following formula for $\kappa(\Gamma^\ast)$:
$$\kappa(\Gamma^\ast)=\frac{\left[ \prod_{j=1}^{k}\left( m_j \right)^{x_j-1}\right]\left[\sum_{j=1}^{k}  \sum_{\tau\in \tau_j} \sigma(\tau)\right]}{n}.$$
}}
\end{remark}


Next, we present some applications of the preceding results.

\noindent
{\it{Application 1, clique--replaced paths}}: Now we consider a particular graph, the path on $k$ vertices  $\Gamma=P_k$, where $k \geqslant 3$ is an integer. For this special case, we 
apply the technique of Remark \ref{alt_appch} in order to obtain the complexity of the graph $\Gamma_{[x_1, \ldots, x_k]}$.  
 For $\Gamma=P_k,$ the directed graph $W$ of Remark \ref{alt_appch} is given as follows.

\vspace{1.5cm} \setlength{\unitlength}{3mm}
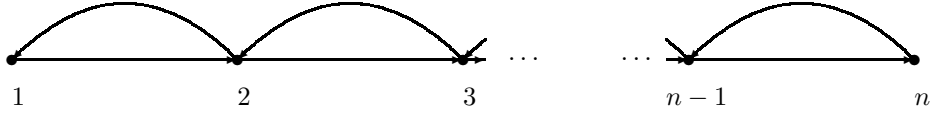
\begin{picture}(0,0)(1.5,9)
\linethickness{0.7pt} 

\put(0,10){\circle*{0.5}}
\put(10,10){\circle*{0.5}}
\put(20,10){\circle*{0.5}}

\put(30,10){\circle*{0.5}}
\put(40,10){\circle*{0.5}}

\put(0,8){$1$}
\put(10,8){$2$}
\put(20,8){$3$}
\put(29,8){$n-1$}
\put(40,8){$n$}

\put(0,10){\vector(1,0){10}}
\put(10,10){\vector(1,0){10}}
\put(20,10){\vector(1,0){1}}
\put(30,10){\vector(1,0){10}}
\put(29,10){\vector(1,0){1}}
\put(22,10){$\ldots$}
\put(27,10){$\ldots$}

\qbezier(30,10)(29.5,10.5)(29,10.9)

\qbezier(0,10)(5,15)(10,10)
\qbezier(10,10)(15,15)(20,10)
\qbezier(20,10)(20.5,10.5)(21,10.9)
\qbezier(30,10)(35,15)(40,10)

\put(1,11){\vector(-1,-1){1}}
\put(11,11){\vector(-1,-1){1}}
\put(21,11){\vector(-1,-1){1}}
\put(31,11){\vector(-1,-1){1}}

\end{picture}\\[.5cm]

Observe that $\tau_1$ contains one directed graph of weight $x_1x_2x_3 \ldots x_{n-1}$, while  
$\tau_n$ contains one directed graph of weight $x_2x_3 \ldots x_n.$ Further, for each $j=2, \ldots, n-1,$ we find that $\tau_j$ contains one directed graph of weight $$(x_2 \ldots x_j)(x_j \ldots x_{n-1}).$$ Consequently for the matrix $\mathbf{S}$ of Remark \ref{alt_appch}, we have 
$$\sum_{j=1}^{k} \det(\mathbf{S}_{(j)})=x_1\ldots x_{n-1} + \sum_{j=2}^{n-1}x_j(x_2 \ldots x_{n-1}) + x_2 \ldots x_n = (x_2 \ldots x_{n-1})\sum_{j=1}^n x_j.$$ 
It now follows that 
\begin{eqnarray*}
\kappa\left(\Gamma_{[x_1, \ldots, x_k]}\right) = \hspace{3.75in}\\
(x_1+x_2)^{x_1-1}\Pi_{j=2}^{n-1}(x_{j-1}+x_j+x_{j+1})^{x_j-1} (x_{n-1}+x_n)^{x_n-1}(x_2 \ldots x_{n-1})\sum_{j=1}^n x_j.
\end{eqnarray*}

\noindent
{\it{Application 2, Cayley's theorem}}: In the case when $\Gamma=K_t$ and $x_1=\cdots=x_t=x$, we have  $\Gamma_{[x_1,\ldots,x_t]}=K_{tx}$.  Moreover,  in the situation of  Theorem \ref{result1} we have:
$n_i=\cdots=n_t=tx-1$, $m_i=\cdots=m_t=tx$,   $\lambda_i=\cdots=\lambda_t=t$ and $\Psi=t^t$.
Substitution into Eq. (\ref{formula1}) yields $$\kappa(K_{tx})=\left(\prod_{i=1}^{t}(tx)^{x}\right)(t^t+0)/(t^t(tx)^2)=(tx)^{tx-2},$$  which is equivalent to Cayley's result. \\

\noindent
{\it{Application 3, complexity of}} $\mathcal{P}(\mathbb{Z}_n)$: Given a natural number $n$,  the {\em divisor graph} $D(n)$ of $n$ is the graph with vertex set $\pi_d(n)=\{d_1, \ldots, d_k\}$, the set of all divisors of $n$, in which two distinct divisors $d_i$ and $d_j$ are adjacent if and only if  $d_i|d_j$ or
$d_j|d_i$.   Let $d_1>d_2>\cdots>d_k$ (evidently $d_1=n$ and $d_k=1$). This shows that (see also \cite[Theorem 2.2]{Mehranian}): 
\begin{equation}\label{eq4} \mathcal{P}(\mathbb{Z}_n)=D(n)_{[\phi(d_1), \ldots, \phi(d_k)]}.\end{equation}
In what follows, we put $\Gamma=D(n)$,  
$n_i=\phi(d_i)-1+\sum_{d_j\in N_\Gamma(d_i)} \phi({d_j})$,
$m_i=n_i+1$,
 $\lambda_i=\frac{m_i}{\phi(d_i)}, \ \ i=1, \ldots, k,$ and 
 $\Phi=\prod_{i=2}^{k-1} \lambda_i, \ \  \Psi=\prod_{i=1}^{k} \lambda_i.$
By using Theorem \ref{result1}, we have the following alternate proof of a result in \cite{MRSS}.  

\begin{corollary}{\rm \cite[Theorem 4.1]{MRSS}}\label{result2}  Let $d_1>d_2>\cdots>d_k$ be the divisors of a positive integer
$n$. With the notation as above, we have 
$$\kappa(\mathbb{Z}_{n})=\prod_{i=1}^{k}m_i^{\phi(d_i)}\Big(\Phi+
\sum_{\Lambda} \det\mathbf{A}_{\Gamma^c}(\Lambda)\lambda_2^{t_2}
\lambda_3^{t_3}\cdots \lambda_{k-1}^{t_{k-1}}\Big)/(\Phi
n^2),$$ where $t_i\in \{0, 1\}$, $2\leqslant i\leqslant 
k-1$, and the summation is over all induced subgraphs $\Lambda$ of
$\Gamma^c\setminus \{d_1, d_k\}$   whose  vertex set $\{d_{i_1}, \ldots,
d_{i_s}\}$ corresponds to 
 $\{i_j| t_{i_j}=0 \}.$  
\end{corollary}
\begin{proof} Using Eq. (\ref{eq4}) and Theorem \ref{result1}, we obtain 
\begin{equation}\label{eq5} \kappa(\mathbb{Z}_n)=\prod_{i=1}^{k}m_i^{\phi(d_i)}\Big(\Psi+
\sum_{\Lambda'} \det \mathbf{A}_{\Gamma^c}(\Lambda')\lambda_1^{t_1}
\lambda_2^{t_2}\cdots \lambda_{k}^{t_{k}}\Big)/(\Psi
n^2),\end{equation} where $t_i\in \{0, 1\}$, $i=1,  \ldots,
k$, and the summation is over all induced subgraphs $\Lambda'$ of
$\Gamma^c$   whose  vertex set $\{d_{i_1}, \ldots,
d_{i_s}\}$ corresponds to 
  $\{i_j| t_{i_j}=0 \}.$  
 Since $\deg_\Gamma(d_1)=\deg_\Gamma(d_k)=k-1$, we obtain 
$\deg_{\Gamma^c}(d_1)=\deg_{\Gamma^c}(d_k)=0$. Thus, if an induced subgraph $\Lambda'$ of
$\Gamma^c$ contains $d_1$ or $d_k$, then  $\det \mathbf{A}_{\Gamma^c}(\Lambda')=0$, while if it does not contain $d_1$ and $d_k$, then the sum $\sum_{\Lambda'} \det \mathbf{A}_{\Gamma^c}(\Lambda)\lambda_1^{t_1}
\lambda_2^{t_2}\cdots \lambda_{k}^{t_{k}}$ is divisible by $\lambda_1\lambda_k$. Hence, we can write 
$$\sum_{\Lambda'} \det \mathbf{A}_{\Gamma^c}(\Lambda)\lambda_1^{t_1}
\lambda_2^{t_2}\cdots \lambda_{k}^{t_{k}}=\lambda_1\lambda_k\sum_{\Lambda} \det \mathbf{A}_{\Gamma^c}(\Lambda)
\lambda_2^{t_2}\cdots \lambda_{k-1}^{t_{k-1}}$$
where the $\Lambda$ run over all induced subgraphs of $\Gamma^c\setminus \{d_1, d_k\}$   whose  vertex set $\{d_{i_1}, \ldots,
d_{i_s}\}$ corresponds to 
  $\{i_j| t_{i_j}=0 \}.$  
 Substituting this in Eq. (\ref{eq5}) and simplifying now yields the
result.
\end{proof}


\section{ Computing the Complexity $\kappa(G)$}
In this section we consider the problem of finding the complexity of power graphs associated with
certain finite groups. 
\subsection {The simple groups $L_2(q)$}  Let $q=p^n\geqslant 4$  for a prime $p$ and some $n\in {\Bbb N}$. 
We are going to find an explicit formula for $\kappa(L_2(q))$. Before we start,  we need some well known facts about the simple groups $G=L_2(q)$, $q\geqslant 4$, which are proven in \cite{Hup}:
\begin{itemize}
\item[{\rm (a)}]   $|G|=q(q-1)(q+1)/k$ and $\mu(G)=\{p, (q-1)/k, (q+1)/k\}$, where $k={\rm gcd}(q-1, 2)$.

\item[{\rm (b)}] Let $P$ be a Sylow $p$--subgroup of $G$. Then  $P$ is an elementary abelian $p$--group of order
$q$, which is a TI--subgroup, and  $|N_G(P)|=q(q-1)/k$.

\item[{\rm (c)}] Let $A\subset G$ be a cyclic subgroup of order $(q-1)/k$. Then $A$ is a TI--subgroup and
the normalizer $N_G(A)$ is a dihedral group of order $2(q-1)/k$.

\item[{\rm (d)}] Let $B\subset G$ be a cyclic subgroup of order $(q+1)/k$. Then  $B$ is a TI--subgroup
and the normalizer $N_G(B)$ is a dihedral group of order $2(q+1)/k$.
\end{itemize}
We recall that a subgroup  $H\leqslant G$  is a {\em TI--subgroup} (trivial intersection subgroup) if for every $g\in G$, either $H^g=H$ or $H\cap H^g=\{1\}$.

\begin{theorem} \label{char2} Let $q=p^n$, with $p$ prime  and $n\in {\Bbb N}$  and let $G=L_2(q)$.  Then we have:
$$\kappa(G)=p^{\frac{(q^2-1)(p-2)}{p-1}}\cdot \kappa\left({\Bbb Z}_{\frac{q-1}{k}}\right)^{q(q+1)/2}\cdot \kappa\left({\Bbb Z}_{\frac{q+1}{k}}\right)^{q(q-1)/2},$$
where $k={\rm gcd}(q-1, 2)$, except exactly in the cases $(p, n)=(2, 1)$, $(3, 1)$.
In particular, we have 
\begin{itemize}
\item[{\rm (1)}]$A_5\cong L_2(5)\cong L_2(4)$ and $\kappa(A_5)=3^{10}\cdot 5^{18}$ {\rm (see \cite{MRSS})}.

\item[{\rm (2)}]$L_3(2)\cong L_2(7)$ and  $\kappa(L_3(2))=2^{84}\cdot 3^{28}\cdot 7^{40}$.

\item[{\rm (3)}]$A_6\cong L_2(9)$ and $\kappa(A_6)=2^{180}\cdot 3^{40}\cdot 5^{108}.$
\end{itemize}
\end{theorem}
\begin{proof}
Let  $q=p^n$, with $p$ prime and $n\in {\Bbb N}$, and $(p, n)\neq (2, 1), (3, 1)$.
As already mentioned,  $G$ contains abelian subgroups
 $P$, $A$ and $B$, of orders $q$, $(q-1)/k$ and $(q+1)/k$, respectively, every distinct pair of their conjugates 
 intersects trivially, and every element
of $G$ is a conjugate of an element in $P\cup A\cup B$.
Let $$G=N_Pu_1\cup \cdots \cup N_Pu_r=N_Av_1\cup \cdots \cup N_Av_s= N_Bw_1\cup \cdots \cup N_Bw_t,$$
be coset decompositions of $G$ by $N_P=N_G(P)$, $N_A=N_G(A)$ and $N_B=N_G(B)$, where
  $r=[G:N_P]=q+1$, $s=[G:N_A]=q(q+1)/2$  and  $t=[G:N_B]=(q-1)q/2$.
Then, we have
\begin{equation} \label{eq-6}
G=P^{u_1}  \cup \cdots \cup P^{u_r}\cup A^{v_1}\cup \cdots \cup A^{v_s}\cup B^{w_1} \cup \cdots \cup B^{w_t}. \end{equation}
Applying Theorem 3.4 (b) in \cite{MRSS} to Eq. (\ref{eq-6}),  we obtain
$$\kappa(G)=\kappa_G(P)^r\cdot \kappa_G(A)^s\cdot \kappa_G(B)^t=\kappa\left({\Bbb E}_{q}\right)^r\cdot \kappa\left({\Bbb Z}_{\frac{q-1}{k}}\right)^s\cdot \kappa\left({\Bbb Z}_{\frac{q+1}{k}}\right)^t,$$
and so by Corollary \ref{cor-epo}, we get 
$\kappa(G)=\left(p^{\frac{q-1}{p-1}(p-2)}\right)^r\cdot \kappa\left({\Bbb Z}_{\frac{q-1}{k}}\right)^s\cdot \kappa\left({\Bbb Z}_{\frac{q+1}{k}}\right)^t.$ 
The result follows.
\end{proof}


\subsection {Extra--special $p$--groups of order $p^3$}
In the sequel, $P$ will be a $p$--group, with $p$   prime. 
We recall below some facts about extra--special groups and other necessary information. We begin with the definition of the extra special groups.
 A $p$--group $P$ is called {\em extra--special} if $Z(P) = [P,P] = \Phi(P)\cong \mathbb{Z}_p$, where   $\Phi(P)$ is the Frattini subgroup of $P$. 
If $P$ is an extra--special $p$--group, then  the order of $P$ is $p^{2n+1}$ for some positive integer $n$. The smallest nonabelian extra--special groups are of order $p^3$.  When 
$p=2$, there are, up to isomorphism, two extra--special $2$--group 
of order $8$, namely, $D_8$ and $Q_8$. The exponent of both of these groups is $p^2=4$.  Furthermore, from 
\cite[Table 1]{MRSS}, we have $\kappa(D_8)=2^4$ and $\kappa(Q_8)=2^{11}$.   

For each odd prime $p$, up to isomorphism, 
 there are just two non--isomorphic extra--special $p$--groups of order $p^3$. The first one has exponent $p$, which is called the {\em Heisenberg group}  and denoted by $H_p$.  In fact, $H_p$ as a subgroup of ${\rm GL}(3,p)$ can be 
 presented in the following way:
$$H_p=\left\{ \left(\begin{array} {lll} 1 & 0 & 0\\ x & 1 & 0\\ z & y & 1\end{array}\right) \ \Bigg | \ x, y, z\in {\rm GF}(p) \right\}.$$
The other one has exponent $p^2$, which is denoted by $A_p$, and contains 
transformations $x\mapsto ax+b$ from $\mathbb{Z}_{p^2}$ to $\mathbb{Z}_{p^2}$, where $a\equiv 1\pmod{p}$ and $b\in \mathbb{Z}_{p^2}$.

The groups $H_p$ and $A_p$  are usually presented as: 
$$H_p=\langle x, y, z \ | \ x^p=y^p=z^p=1, [x,y]=z, [x,z]=[y,z]=1\rangle,$$
and
$$A_p=\langle x, y\ | \ x^{p}=y^{p^2}=1, y^x=y^{p+1}\rangle.$$
\begin{theorem}\label{extra-p} Let $p$ be an odd prime. Then, we have:
\begin{itemize}
\item[{\rm (a)}] $\kappa(H_p)=p^{(p-2)(p^2+p+1)}$.
\item[{\rm (b)}] $\kappa(A_p)=p^{2p^3-p-5}$.
\end{itemize}
\end{theorem}
\begin{proof} (a) Clearly,  we have $$H_p=\bigcup_{j=1}^{p^2+p+1}C_j,$$
where  $C_j\subset H_p$ is a subgroup of order $p$, and $C_i\cap C_j=1$ for $i\neq j$.
Now, by Theorem 3.4 (b) in \cite{MRSS}, we obtain 
$$\kappa(H_p)=\prod_{j=1}^{p^2+p+1} \kappa(C_j)=\prod_{j=1}^{p^2+p+1} p^{p-2}=p^{(p-2)(p^2+p+1)}, $$ 
as desired.

(b) In this case, we have $$A_p=\bigcup_{j=1}^{p+1}B_j,$$
where  $B_j\subset A_p$ is a subgroup of order $p^2$, and $B_i\cap B_j=Z(A_p)$ for $i\neq j$.
Therefore, the power graph of $A_p$ has the following form
$${\cal P}(A_p)=K_p\vee \left[(p+1)K_{p^2-p}\right].$$
It follows by Lemma \ref{elementary0} that  the eigenvalues of Laplacian matrix  $\mathbf{L}_{{\cal P}(A_p)}$ are:
$$p^3, \underbrace{p^2, \ p^2, \ \ldots, \ p^2}_{p-1}, \  \underbrace{p^2, \  p^2, \ \ldots, \ p^2}_{p^3-2p-1},   \  \underbrace{p, \  p, \ \ldots, \ p}_{p}, \ 0. $$
 Using Eq. (\ref{eq1}), we get 
$\kappa(A_p)=p^{2p^3-p-4}$,
as required.
\end{proof}


\subsection{Frobenius groups} Suppose $1\subset H\subset G$ and  $H\cap H^g=1$ whenever $g\in  G\setminus H$. Then $H$ is a {\em Frobenius complement} in $G$. A group which contains a Frobenius complement is called a {\em Frobenius group}.
A famous theorem of Frobenius asserts that in a Frobenius group $G$ with a Frobenius complement $H$, the set $$F=\left(G\setminus \bigcup_{g\in G} H^g\right)\cup\{1\},$$
is a normal subgroup of $G$ and $G=FH$, $F\cap H=1$. We call $F$ the {\em Frobenius kernel} of $G$.
\begin{theorem} \label{lm-3} Let $G$ be a Frobenius group,
$H$ a Frobenius complement and $F$ the Frobenius kernel
corresponding with $H$. Then, we have:  $$\kappa(G)=
\kappa_G(F)\kappa_G(H)^{|F|}.$$ In particular, if $G$ is a
nonabelian group of order $pq$, where $p<q$ are primes, then
$\kappa(G)=q^{q-2}p^{(p-2)q} .$
\end{theorem}
\begin{proof} Let $G$ be a Frobenius group, let $H$ be its Frobenius complement
and $F$ its Frobenius kernel. Then $G$ can be written as the union
of its subgroups:
$$G=F\cup \bigcup_{g\in F}H^g.$$
 Again, it follows from Theorem 3.4 (b) in \cite{MRSS} that
$$\kappa(G)=\kappa_G(F)\prod_{g\in
F}\kappa_G(H^g)=\kappa_G(F)\kappa_G(H)^{|F|},$$ as required.
\end{proof}


\begin{center}
 {\bf Acknowledgments }
\end{center}
This work was done while the second author had a visiting position at the
Department of Mathematical Sciences, Kent State
University, USA. He is thankful for the hospitality of the Department of Mathematical Sciences of KSU.
The research of the first author is supported by the Natural Sciences
and Engineering Research Council of Canada under grant number
RGPIN/6123--2014.

\end{document}